\documentclass[reqno,11pt]{amsart}
\usepackage{amsmath, latexsym, amsfonts, amssymb, amsthm, amscd, stmaryrd}
\usepackage{graphics,epsf,psfrag, color}
\numberwithin{equation}{section}
\allowdisplaybreaks
\usepackage{bbm, dsfont}
\usepackage{mathtools}
\usepackage{xcolor}

\newtheorem{theorem}{Theorem}[section]
\newtheorem{lemma}[theorem]{Lemma}
\newtheorem{proposition}[theorem]{Proposition}

\newtheorem{con}[theorem]{Conjecture}

\newcommand{\RN}[1]{%
  \textup{\uppercase\expandafter{\romannumeral#1}}%
}

\makeatletter
\def\captionfont@{\footnotesize}
\def\captionheadfont@{\scshape}

\long\def\@makecaption#1#2{%
  \vspace{2mm}
  \setbox\@tempboxa\vbox{\color@setgroup
    \advance\hsize-6pc\noindent
    \captionfont@\captionheadfont@#1\@xp\@ifnotempty\@xp
        {\@cdr#2\@nil}{.\captionfont@\upshape\enspace#2}%
    \unskip\kern-6pc\par
    \global\setbox\@ne\lastbox\color@endgroup}%
  \ifhbox\@ne 
    \setbox\@ne\hbox{\unhbox\@ne\unskip\unskip\unpenalty\unkern}%
  \fi
  \ifdim\wd\@tempboxa=\z@ 
    \setbox\@ne\hbox to\columnwidth{\hss\kern-6pc\box\@ne\hss}%
  \else 
    \setbox\@ne\vbox{\unvbox\@tempboxa\parskip\z@skip
        \noindent\unhbox\@ne\advance\hsize-6pc\par}%
\fi
  \ifnum\@tempcnta<64 
    \addvspace\abovecaptionskip
    \moveright 3pc\box\@ne
  \else 
    \moveright 3pc\box\@ne
    \nobreak
    \vskip\belowcaptionskip
  \fi
\relax
}
\makeatother
\def\writefig#1 #2 #3 {\rlap{\kern #1 truecm
\raise #2 truecm \hbox{#3}}}

\renewcommand{\tilde}{\widetilde}
\renewcommand{\hat}{\widehat}

\newcommand{\cF}{\ensuremath{\mathcal F}}

\newcommand{\cP}{\ensuremath{\mathcal P}}

\newcommand{\N}{{\ensuremath{\mathbb N}} } 
 
\renewcommand{\P}{{\ensuremath{\mathbb P}} }

\newcommand{\bbZ}{{\ensuremath{\mathbb Z}} }

\newcommand{\eps}{\varepsilon}

\newcommand{\1}[1]{\mathbf{1}_{\left\lbrace   #1   \right\rbrace }}

\renewcommand{\t}{\tau}
\renewcommand{\l}{\lambda}
\renewcommand{\a}{\alpha}

\renewcommand{\a}{\alpha}
\newcommand{\gb}{\beta}
\renewcommand{\d}{\delta}

\newcommand{\w}{\omega}

\newcommand{\bP}{{\ensuremath{\mathbf P}} }
\newcommand{\bE}{{\ensuremath{\mathbf E}} }

\newcommand{\E}{\mathbb{E}}
\newcommand{\Z}{\mathbb{Z}}
\newcommand{\R}{\mathbb{R}}

\newcommand{\g}{\gamma}
\renewcommand{\b}{\beta}

\begin{document}
\title{Directed Polymer for very heavy tailed random walks }

\author[Roberto Viveros]{Roberto Viveros}
 \address{Roberto Viveros \hfill\break
IMPA\\
Estrada Dona Castorina, 110\\ Rio de Janeiro 22460-320 \\ Brazil.}
\email{rviveros@impa.br}


\begin{abstract}
In the present work, we investigate the case of Directed Polymer
in a Random Environment (DPRE), when the increments of the random walk are heavy-tailed with tail-exponent equal to zero ($\bP[|X_1|\ge n]$ decays slower than any power of $n$). This case has not yet been studied in the context of directed polymers and present key differences with the simple symmetric random walk case and the cases where the increments belong to the domain of attraction of an $\alpha$-stable law, where $\alpha \in (0, 2]$. We establish the absence of a very strong
disorder regime - that is, the free energy equals zero at every temperature - for every disorder distribution. We also prove that a strong disorder regime (partition function converging to zero at low temperature) may exist or not depending on finer properties of the random walk: we establish non-matching necessary and sufficient conditions for having a phase transition from weak to strong disorder. In particular our results imply that for this directed polymer model, very strong disorder is not equivalent to strong disorder, shedding a new light on a long standing conjecture concerning the original nearest-neighbor DPRE. 
\end{abstract}

\maketitle


\section{Introduction}
Directed polymer in random environment is a model for elastic
molecules interacting with random impurities. It appeared originally in the physics literature in the study
of the interface for the Ising model \cite{huse} and has become an interesting subject of study for many authors ever since (see \cite{review, saintf} for a review on the matter). 

Loosely speaking, the model consists on a random walk (of law denoted by $\bP$) on the integer lattice $\Z^{1+d}$, which stretches in the time direction, and interacts with a random space-time environment (of law denoted by $\P$) whose intensity is parameterized by some constant $\b \geq 0$ (inverse temperature). Given a fixed realization of the environment (sometimes also referred as the disorder), new weights are assigned to the walks. The $\bP$-expectation of this weight is the partition function of the system and the Liapunov exponent of this expectation is the quenched free energy (see the formal definitions later). 

Most of the literature concerning the study of directed polymers associates it with a simple symmetric random walk \cite{comets,bolthausen,cometsyoshido,hubert} or when the distribution of the increments belongs to the domain of attraction of an $\alpha$-stable law for some $\a \in (0, 2]$ \cite{comets long jump, mat, wei}.

It is known that there is a phase transition both in the limit of the partition functions and also in the free energy. In particular, there is a critical value $\b_c$ below which the sequence of normalized partition functions has a strictly positive	limit $\P$-a.s. (weak disorder), while above $\b_c$ the limit is zero $\P$-a.s. (strong disorder). Moreover, there is a second critical value $\bar{\b_c}$ below which the quenched free energy is equal to its annealed counterpart, while above it is strictly smaller than it (very strong disorder).  It is not hard to see that  $\b_c \leq \bar{\b_c}$ and a question of interest is whether these critical points are different.

Informally, in the weak disorder regime, the polymer paths are globally not affected by the environment, for instance, showing diffusivity when $\bP$ is the SRW while displaying localization phenomena and superdiffusivity in the strong disorder regime.

So far, it has been shown that $\b_c = \bar{\b_c} = 0$ for the nearest-neighbor directed polymer on $\Z^{d+1}$ for $d = 1$ in [11] and $d = 2$ in	[15], and for the long-range directed polymer with underlying random walks in the domain of attraction of an $\alpha$-stable law for some $\alpha \in (1, 2]$ in [16] for $d = 1$. A second moment computation of the partition function shows that $\b_c > 0$ whenever the random walk is transient \cite{bolthausen} (see \cite{viveros} for a study of the phase diagram when the environment displays a heavier tails), but the question of whether these two critical points coincide remains open whenever $\b_c > 0$. It has been conjectured that $\b_c = \bar{\b_c}$.

Our aim in this paper is to examine the case when the exponent of the distribution of the increments is equal to one. Specifically, for $d = 1$, assuming that the random walk is defined as $S_n = X_1 + ... + X_n$, where 
$\left\lbrace X_i: i \in \N\right\rbrace  $ is a { sequence  of i.i.d. random variables (also known as the increments) taking values in $\bbZ$. We assume that the increments have symmetric distribution and that  for $n\in \Z\setminus\{0\}$ we have}
\begin{equation} \label{1}
\bP \left[X_1 = n \right] =:K(n) = \frac{L(n)}{n}, 
\end{equation} 
where $L(\cdot)$ is a slowly varying function at $\pm \infty$. 

Interestingly, the phenomenology in this case is different than what has been seen before. We show that the quenched free energy is equal to the annealed free energy at every temperature ($\bar{\b_c} = \infty$) and that under some additional hypothesis, the strong regime is non-trivial ($ \b_c < \infty$), proving that the conjecture cannot hold in complete generality.

The first result is inspired by the work in \cite{alexzyg} in which an analogous result is proven, for the pinning model: the quenched critical point and the annealed one coincide for any given value of $\b \geq 0$ when the law of the renewal process $\t$ has loop exponent one, i.e. 
\begin{equation}
\P \left[ \t = n\right] = \frac{L(n)}{n}. 
\end{equation}
for some slowly varying function $L$. We also mention the work in \cite{julien} where low disorder relevance  is proven, in the hierarchical pinning model at every temperature, in the $b = s$ case. These are analogous notions of very strong disorder for the pinning model and the hierarchical pinning model respectively. 

In our second and third results we prove a sufficient and a necessary condition on $\bP$ and $\P$ for $ \b_c < \infty$. This has no analogous version for the pinning model, as there is no notion of weak disorder developed in that context so far.

The organization of the rest of the introduction goes as follows: In the next section we give the formal definition of the model and state already known facts. Then we present our results and give some comments on the  extra hypothesis needed and methods used in the proofs.   

%
%
%

\subsection{Polymer measure}
On the space $ \left( {\left( \Z^d\right) }^{\N}, \cP(\Z^d)^{\otimes \N}\right) $ of sequences $S:=(S_n)_{n\geq 0}$, let $\bP$ be a probability measure that satisfies:
\begin{equation}
\begin{split}\label{first}
&S_0 = 0,\\
& \{S_{n} - S_{n-1}\}_{n\geq 1}\text{ is an i.i.d. sequence.}
\end{split}
\end{equation}
We say that $\bP$ is a random walk on $\Z^d$. Most of the results in the literature assumes that $\bP$ is the law of the nearest-neighbor symmetric random walk:
\begin{equation}
\bP [S_1 = e_j] = \bP [S_1 = - e_j] = \frac{1}{2d},
\end{equation}
where $\{e_1,...,e_d\}$ is the canonical basis of $\R^d$, but the results stated in this sub-section are true in the general setting \eqref{first}.

Independently, also consider a set of i.i.d. random variables $\w:= \{\w_{n,z}: n \in \N, z \in \Z^{d}\}$, called \textit{ the environment}, defined on a probability space $(\Lambda, \cF, \P)$, that satisfies,
\begin{equation}
\E \left[ \exp(\beta \w_{n,z}) \right] < \infty, 
\end{equation}
for any $\b \in \R$. The \textit{polymer measure} $\bP_{N}^{\beta,\eta}$ is the probability measure in $ \left( {\left( \Z^d\right) }^{\N}, \cP(\Z^d)^{\otimes \N}\right) $ describe by its Radon-Nikodym derivative with respect to $\bP$: For a fixed value of $\b$ (called the inverse temperature) and $N \in \N$ we let
\begin{equation} \label{poly}
\dfrac{\mathrm{d} \bP_{N}^{\beta,\w}}{\mathrm{d} \bP}(S) = \dfrac{1}{Z_{N}^{\beta,\w}} \exp \left(\beta \sum_{n=1}^{N} \w_{n,S_n} \right). 
\end{equation}
The positive normalization factor $Z_{N}^{\beta,\w}$ (called the \textit{partition function}) makes $\bP_{N}^{\beta,\w}$ a probability measure. Consider the re-normalized partition function 
\begin{equation}
W_N^{\beta,\eta} := \frac{Z_{N}^{\beta,\eta}}{\E \left[ Z_{N}^{\beta,\eta}\right] }. 
\end{equation}
In \cite{bolthausen}, Bolthausen observed that the sequence $\left\lbrace W_N, \mathcal{G}_N \right\rbrace_{N \in \N}$ is a positive martingale, where $\left\lbrace \mathcal{G}_N \right\rbrace _{N \geq 0}$ is the filtration defined by $\mathcal{G}_N := \sigma\{\w_{n,z}: 1\leq n \leq N, z \in \Z \} $. By the classical martingale theory, it follows that the limit 
\begin{equation}
W_{\infty}^{\beta,\w} := \lim_{N \to \infty} W_{N}^{\beta,\w},
\end{equation}
exists $\P$-\textit{a.s.} and is a non-negative random variable. Moreover, the event $\{W_{\infty}^{\beta,\eta} = 0\}$ belongs to the tail $\sigma$-field of $\{\mathcal{G}_N, N \geq 0\}$. Hence, by Kolmogorov's $0-1$ Law,
\begin{align}\label{dichotomy}
&\P \left\lbrace  W_{\infty}^{\beta, \w}  > 0  \right\rbrace \in \{0,1\}. 
\end{align}
Following standard terminology we say that we have \textit{weak disorder} if $W_{\infty}^{\beta} > 0$ $\P$-\textit{a.s.} and \textit{strong disorder} if $W_{\infty}^{\beta} = 0$ $\P$-\textit{a.s.} In \cite{comets}, it is shown that there exists a critical value $\b_c \in [0,\infty]$, depending possibly on the environment distribution, such that there is weak disorder for $\b \in [0,\b_c)$ and strong disorder for $\b > \b_c$. The \textit{quenched free energy} is defined as
\begin{equation} \label{freenergy2}
F(\beta) := \lim_{N \to \infty} \dfrac{1}{N}   \log  Z_{N}^{\beta,\w} = \lim_{N \to \infty} \dfrac{1}{N}  \E \log  Z_{N}^{\beta,\w}.
\end{equation}
It is known that this limit exists and does not depend on $\w$ (see \cite[Proposition 2.5]{cometsyoshido} for the nearest-neighbor case and \cite{bates} for the general case), except on a set of measure zero. By Jensen's Inequality we have that $F(\b) \leq \l(\b) $, where $\l(\b) := \log \E \exp (\b \w)$ (the \textit{annealed free energy}). Also, it is not hard to see that
\begin{equation} \label{13}
F(\b) < \l(\b) \implies \lim_{N \to \infty}   W_{N}^{\beta,\w} = 0 \quad \P-\textit{a.s.}
\end{equation}
Thus, the case $p(\b):=F(\b) - \l(\b)  < 0$ is called the\textit{ very strong disorder}. As a function, $p(\cdot)$ is continuous and non-increasing. There is a critical value $\bar{\b_c}$ such that $p(\b) = 0$ if $\b \in [0,\bar{\b_c}]$ and $p(\b) < 0$ if $\b > \bar{\b_c}$. As noted before, $\b_c \leq \bar{\b_c}$.

It is conjectured that there is  no intermediate phase between weak disorder and very strong disorder (i.e., $\b_c = \bar{\b_c}$) but so far this has only been proved for the simple symmetric directed polymer on dimensions $d = 1$ and $d = 2$ in which $\b_c = \bar{\b_c} = 0$ \cite{hubert} and for the long-range directed polymer where the underlying random walk in the domain of attraction of an $\alpha$-stable law for some $\alpha \in (1, 2]$ in \cite{wei} for $d = 1$.

\subsection{The results}

For the rest of the paper, we assume that the law $\bP$ of the random walk satisfies \eqref{1}.
\begin{theorem} \label{T1}
	Consider the polymer measure \eqref{poly} and assume that the distribution of the increments satisfies \eqref{1} and that $K(n) > 0$ for all $n \in \Z$ then, 
	\begin{equation}
	p(\b) = 0,
	\end{equation}
	for all $\b \in \R$, which implies that there is no very strong disorder regime.
\end{theorem}
The extra assumption $K(n) > 0$ appears only in Lemma 2.1 and is not really necessary. It is used to avoid technical details that are not part of the main ideas of the proof.

The result of the first theorem contrasts with the cases that have been studied before, in particular, in the $\a$-stable case, $p(\b)<0$ for sufficiently large $\b$ \cite[Proposition 5.1]{comets long jump}. The next result gives a sufficient condition for which $\b_c < \infty$ which means there is a strong disorder phase. Important quantities here are the entropy $-\sum_{ n\in  \bbZ} K(n) \log K(n)$  of the walk and the mass on the essential supremum of the marginal distribution of $\w_{n,z}$.
\begin{theorem} \label{T2}
	If the distributions of the increments and the environment satisfy
	\begin{equation} \label{entropy}
	\b \l'(\b) - \l(\b) > \sum_{ n\in  \bbZ} K(n) \log \frac{1}{K(n)},
	\end{equation}
	then $${ W^{\beta,\omega}_\infty=0 }\quad  \P-\textit{a.s.}$$
	
	{ \noindent In particular if $\lim_{\b \to \infty}  \b \l'(\b) - \l(\b) = \infty$ then
		\begin{equation} \label{implik}
		\sum_{ n\in  \bbZ} K(n) \log \frac{1}{K(n)}<\infty \quad \Rightarrow 
		\quad  \gb_c <\infty.
		\end{equation}}
\end{theorem}
Note that the condition \eqref{entropy} appears in \cite[Proposition 5.1]{comets long jump} (which studies the case of polymer based on $\alpha$ stable walks) as a sufficient condition to have very strong disorder  ($p(\beta)<\infty$). However here very strong disorder cannot hold in our case (since it would contradict Theorem \eqref{T1}) and the criterion \eqref{entropy} emerges from a proof which is of a different nature than the (fractional moment based) one in \cite[Proposition 5.1]{comets long jump}. 

Setting $s = \mathrm{ess\ sup}\{\omega \}$, we have that $\log \frac{1}{\P\left[\eta = s\right]} > \sum_{ n\in  \bbZ} K(n) \log \frac{1}{K(n)}$ implies that $\b_c < \infty$ as 
\begin{equation}
\lim_{\b \to \infty} \b \l'(\b) - \l(\b)= \log \frac{1}{\P\left[\eta = s\right]}.
\end{equation}
This known property of the exponential moments is proven in the Appendix for completeness. (Lemma \ref{sup}).
The assumption $\lim_{\b \to \infty}  \b \l'(\b) - \l(\b)=\infty$ is equivalent to say that $\omega$ is either unbounded or almost surely does not attain its essential supremum.


{ {
		We note that the condition  $\lim_{\b \to \infty}  \b \l'(\b) - \l(\b) = \infty$ is necessary to have \eqref{implik}}.
	To illustrate our point let us consider the 
	case of the Bernoulli environment with parameter $p$.{ Then}
	there is weak disorder for all $\b$, if 
	$p$ is sufficiently close to one. More specifically, as shown in \cite{comets long jump}, a sufficient condition for which the sequence of polymer measures $W_N^{\b,\w}$ is uniformly bounded in $\mathcal{L}^2$ for all $\beta$ (which implies weak disorder) is that $$p > \bP\otimes\bP' \left[\exists n\geq 1: S_n = S'_n \right], $$ where $S, S'$ are two independent walks.  }

Assuming that the environment is unbounded, Theorem 1.2 permits to conclude that if for some  $\a < -1$, 
\begin{equation}\label{112}
K(n)  \leq \frac{(\log \log n)^\alpha}{n (\log n)^2},
\end{equation}
for all $n$ sufficiently large, the polymer presents a strong disorder phase. { More importantly it provides an example of a directed polymer model for which the two critical points do not coincide ($\b_c < 
	\bar \b_c $). To our knowledge, the existence of such a setup was not predicted in the literature, and while it is not invalidating the conjecture concerning the nearest neighbor model, it sheds a new light on it.
}

In opposition, in the next theorem, we show that under some extra assumptions, if $\alpha > 1$ in \eqref{112}, then there is no strong disorder phase. 
\begin{theorem} \label{T3}
	Under the following conditions on the law of the increments:
	\begin{itemize}
		\item[(a)]  $K(\cdot)$ \text{ is unimodal and symmetric around 0,}
		\item[(b)] For some  $\a < -1$, \begin{equation}
		K(n)  \geq \frac{(\log \log n)^\alpha}{n (\log n)^2},
		\end{equation}
		for all $n$ sufficiently large,
		\item [(c)] and
		\begin{equation}
		\frac{\bP\left[X_1 \in (s_n, 2n s_n) \right]}{\bP\left[X_1 \geq s_n \right]} \leq \frac{1}{n^\g},
		\end{equation}
		where $\g > \dfrac{1}{2}$ and
		\begin{equation}
		s_n := \min \left\lbrace s  \in \N: \bP\left[X_1 \geq s \right] \leq \frac{\left(\log n \right)^2 }{n}\right\rbrace, 
		\end{equation}
		for all $n$ sufficiently large,
	\end{itemize} 
	then, $\b_c = \infty$.
\end{theorem}
Condition (c) might seem artificial at first sight but it is satisfied by most distribution with sufficiently regular tails, as $\bP\left[X_1 \geq n \right] = \frac{L(\log n)}{(\log n)^\a}$ where $\a \leq 1$ or $\bP\left[X_1 \geq n \right] = \frac{L(\log \log n)}{(\log \log n)^\b}$ where $\b > 0$ and $L$ a slowly varying function.

%
%
%

{
	
	\subsection{Conjecture and future {research directions}}\label{higherdim}
	

	At the present moment we are not able to answer whether a strong disorder phase  exists if $\a \in [-1,1]$ and $K(n) \asymp \frac{c (\log \log n)^\a}{n (\log n)^2}$ although we believe that the condition \eqref{entropy} on the entropy might be necessary to the existence of the strong disorder phase.{ Let us make this point more precise}. 
	
	\begin{con}
		Assuming that the environment is unbounded from above, we have {the  following equivalence
			\begin{equation}
			\b_c < \infty  \quad   \Leftrightarrow \quad \sum_{n\geq 1} K(n) \log \frac{1}{K(n)} < \infty.
			\end{equation}}
	\end{con}
	
	\subsection{Acknowledgment}
	The author is very grateful to his PhD advisor Hubert Lacoin for suggesting this problem and for very fruitful discussions during his PhD program. 	
	
	%
	%
	
	\section{Lower bound for the free energy}
	\textit{Idea of the proof}. 
	
	{As we said before, our proof shares some ideas with \cite{alexzyg}. Specifically, since $\bP \left[ X_1 \geq n\right] $ is a slowly varying function of $n$, the longest of the first $m$ excursions typically has length greater than any power of $m$. This enables the polymer to travel further distances, avoiding some regions of insufficiently unfavorable values at low cost. With this in mind, we partition the environment into rectangles of size $N \times 2N^2$, where $N$ is a scaling factor and restrict attention to the ones whose higher values contributes more to the partition function. Roughly speaking, the partition function, when restricted to a \textit{good} rectangle, has a value higher than some appropriate threshold. Further we will lower bound the partition function by considering paths that only travel through these good rectangles. In Lemma \ref{lemma0} we will lower bound the probability of a path to stay inside a rectangle and in Lemma \ref{lemma1}, we control the cost of jumping to a good rectangle. In the proof we make an energy-entropy balancing of the paths that travel only through good rectangles. }
	
	\begin{proof} [Proof of Theorem 1.1]
		Fix  $\eps >0$ arbitrarily small and let $N = N(\b, \eps) \in \N$ be a scaling factor whose value is defined later. Consider the following collection of disjoint rectangles $\cup_{(i,j) \in \Z^2} R_{i,j} = \Z^2 $, each one of size $N \times 2N^2$, defined as
		\begin{equation}
		R_{i,j} := \left\lbrace(x,y) \in \Z^2: iN + 1\leq x \leq (i+1)N, (2j-1)N^2 \leq y < (2j+1)N^2 \right\rbrace. 
		\end{equation}
		In order to lower bound the free energy, we consider only paths that visit rectangles which contribute the most to the partition function. Consider the following restricted version of the normalized partition function to the rectangle $R_{i,j}$:
		\begin{equation}
		\tilde{W}_N(i,j) := \bE^{2jN^2} \left[ \exp \left( \sum_{k= 0}^{N-1}  \b \w_{iN+ k+1,2jN^2 + S_k} - \l(\b) \right) \Bigg| S \in \mathcal{A}_N \right],
		\end{equation} 
		where we define $\mathcal{A}_N$ as the event,
		\begin{equation} \label{A}
		\mathcal{A}_N := \left\lbrace \left(S_k \right)_{k=0}^N:S_{N-1} = S_0,  | S_k - S_0 | < N^2 \text{ for } 0\leq k <N  \right\rbrace. 
		\end{equation}
		Paths considered in the expectation above start at $(iN+1, 2jN^2)$ and remain inside the rectangle until ending up at the vertex $((i+1)N, 2jN^2)$. Notice that by the i.i.d. structure of the environment, $\left\lbrace\tilde{W}_N(i,j): (i,j) \in \Z^2 \right\rbrace$ is an i.i.d. collection of random variables with
		\begin{equation}
		\E \left[\tilde{W}_N(i,j) \right] = 1.
		\end{equation} 
		Depending on the environment's realization, we say that a rectangle $R_{i,j}$ is \textit{$\eta-$\text{good}} when 
		\begin{equation}
		\tilde{W}_N(i,j) \geq \eta,
		\end{equation}
		for some constant $\eta >0$. Let $p_\eta$ be the probability of a rectangle to be $\eta-$good. Given a realization of the environment, let us define the random sequence $\left\lbrace J_{-1}, J_{0}, J_{1},...\right\rbrace $ inductively: Let $J_{-1}=0$ and for $i \geq 0$, 
		\begin{align}
		J_i = \min \left\lbrace j> J_{i-1}: R_{i,j} \text{ is $\eta$-good}\right\rbrace.  
		\end{align}
		We lower bound the partition function $W_{Nm}^{\b, \w}$ by considering trajectories that only visit good rectangles. Specifically, let us consider the trajectories $\left(S_k\right)_{k=0}^{Nm}$ that belong to $\Xi_{Nm}$ where
		\begin{equation}
		\Xi_{Nm} := \left\lbrace\left(S_k\right)_{k=0}^{Nm}: S_{iN+1}=S_{(i+1)N}=2 J_i N^2, |S_{iN+k}- S_{iN}|< N^2, \text{ for } 0\leq i < m, 1\leq k < N  \right\rbrace. 
		\end{equation}
		In other words, when considering the graph of these paths, in $\Z^2$, they do the following
		\begin{itemize}
			\item Starting from (0,0), they jump to the site $(1,2J_0N^2)$ and remain inside $R_{0,J_0}$ until the ending up at the site $(N,2J_0N^2)$.
			\item Inductively for $1 \leq i < m$, after visiting the last site of $R_{i-1,J_{i-1}}$, they jump to the site $(iN+1, 2J_i N^2)$ and remain inside $R_{i,J_i}$ until ending up at the site $((i+1)N,2J_iN^2)$.
		\end{itemize}
		\begin{figure}
			\includegraphics[scale=0.5]{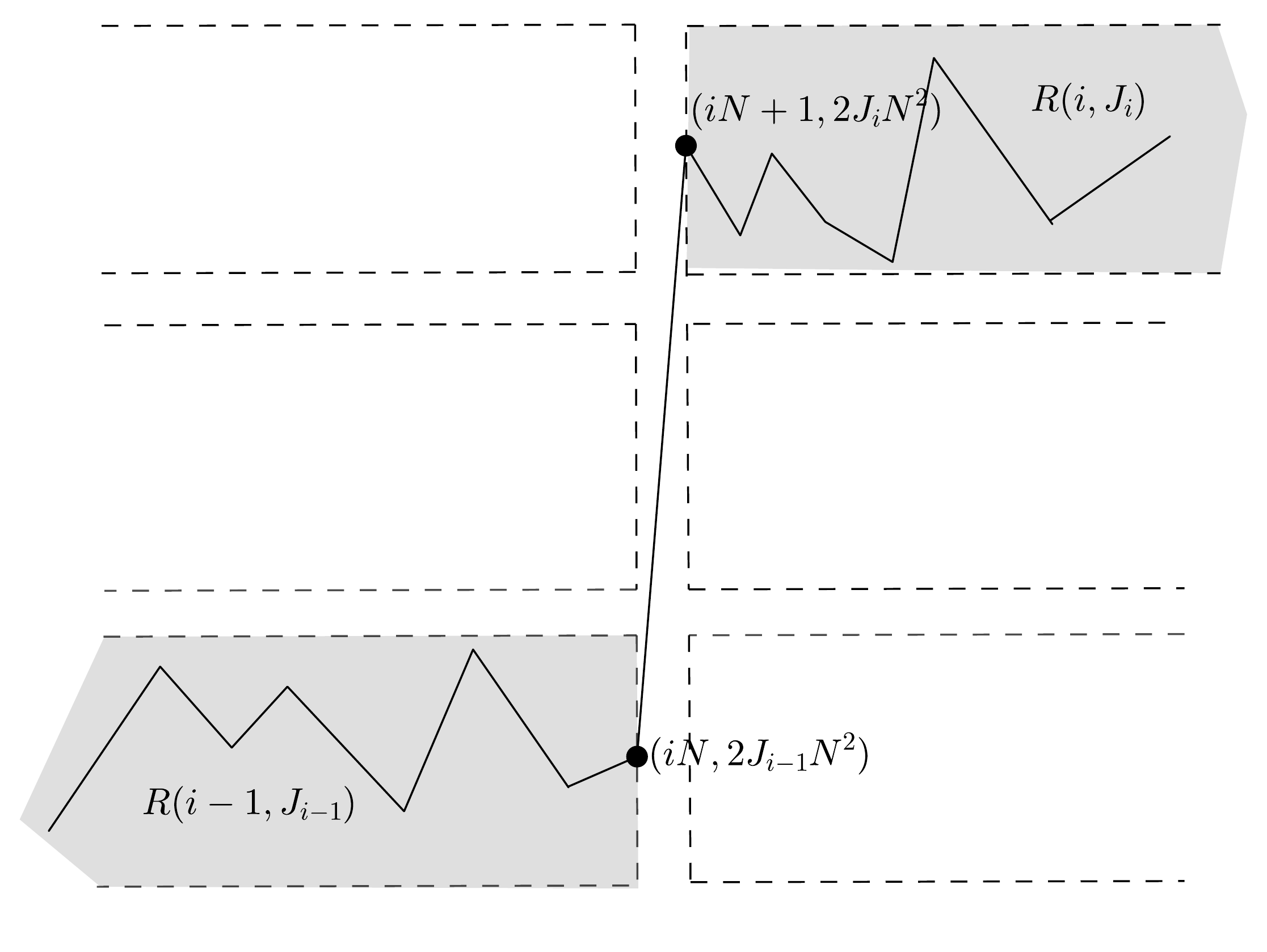}
			\caption{A path that belongs to $\Xi_{Nm}$, after visiting the last site of the good rectangle $R(i-1, J_{i-1})$, jumps to the site $(iN, 2J_i N^2)$ of the first good rectangle $R(i,J_i)$ from the next column.} 
			\label{fig1}
		\end{figure}
		Let $W_{Nm}^{\b, \w}\left(\Xi_{Nm}\right)$ be the partition function restricted to the trajectories that belong to $\Xi_{Nm}$. By the Markov Property
		\begin{align}
		W_{Nm}^{\b, \w}\left( \Xi\right) = \prod_{i=0}^{m-1} \bP \left[X_1 = 2(J_i-J_{i-1}) N^2\right] \tilde{W}(i, J_i) \bP \left[\mathcal{A}_N\right].
		\end{align}
		We then have that
		\begin{equation}
		\begin{split} \label{5}
		\frac{1}{Nm} \log W_{Nm}^{\b, \w} &\geq \frac{\log \eta }{N} + \frac{1}{Nm} \sum_{i=0}^{m-1}  \log K(2(J_i - J_{i-1})N^2) + \frac{\log \bP \left[ \mathcal{A}_N \right]}{N}.
		\end{split}
		\end{equation}
		Letting $m \to \infty$, the left hand side of \eqref{5} converges to the free energy. Notice that since the events 
		\begin{equation}
		\left\lbrace R_{i,j} \text{ is }\eta-\text{good}: (i,j) \in \Z^2 \right\rbrace
		\end{equation}
		are independent, $\left\lbrace J_i - J_{i-1} - 1\right\rbrace_{i \geq 0} $ is an i.i.d. collection of random variables. Therefore, by the Law of Large Numbers,
		\begin{equation}
		\lim_{m\to \infty} \frac{1}{m} \sum_{i=0}^{m-1}  \log K(2(J_i - J_{i-1})N^2) =  \E\left[  \log K(2J_0 N^2)\right] = \E\left[  \log \frac{L(2J_0 N^2)}{2J_0 N^2}\right] . 
		\end{equation}
		Let
		\begin{equation}\label{SVF}
		C_{L,\eps} := \inf \left\lbrace x^\eps L(x): x\geq K \right\rbrace > 0, 
		\end{equation}
		for some $K = K(\eps) > 0 $ sufficiently large. Then, assuming $2N^2 \geq K$ we have, by Jensen's Inequality,
		\begin{equation}
		\begin{split}
		\E\left[  \log \frac{L(2J_0 N^2)}{2J_0 N^2}\right] &\geq \E\left[  \log  \frac{C_{L,\eps}  }{(2J_0 N^2)^{1+\eps}}\right] \geq \log \frac{  C_{L,\eps}  }{(2 N^2)^{1+\eps} } - (1+\eps)\log \E\left[ J_0 \right].
		\end{split}
		\end{equation} 
		As $J_0-1$ is a geometric random variable with parameter $p_\eta$, we have that $\E\left[ J_0 \right] = \frac{1}{p_\eta} + 1$. Then,
		\begin{align} 
		p(\b) &\geq \frac{\log \eta}{N}  +\frac{1}{N} \log \frac{  C_{L,\eps}  }{(2 N^2)^{1+\eps} } - \frac{(1+\eps)}{N} \log \left( \frac{1}{p_{\eta}} + 1\right) + \frac{\log \bP \left[ \mathcal{A}_N\right]}{N}.  \label{final}
		\end{align}
		
		Let us state the following two lemmas, whose proofs are presented at the end of the section. The first one is a straightforward lower bound for $\bP \left[ \mathcal{A}_N\right]$. The second one is more subtle and shows that we can choose a suitable value for $\eta$ such that it compensates the cost $p_\eta$ of the jump. We use these to bound $\frac{\log \bP \left[ \mathcal{A}_N\right]}{N}$ and $\frac{\log \eta}{N} -  \frac{(1+\eps)}{N} \log \left( \frac{1}{p_{\eta}} + 1\right)$ respectively.
		
		\begin{lemma} \label{lemma0}
			With $\mathcal{A}_N$ defined in \eqref{A} we have
			\begin{equation}
			\lim_{N \to \infty} \frac{\log \bP \left[ \mathcal{A}_N\right]}{N} = 0.
			\end{equation}
		\end{lemma}
		\begin{lemma} \label{lemma1}
			There exists $\eta\in [1/2, e^{C_{\beta}N}]$ such that $p_{\eta}\ge \frac{c }{\eta (2+ \log \eta)^2}$, where $c$ and $C_\b$ are constants, the last one depending only on $\b$. 	
			%
			%
		\end{lemma}
		
		Let us finish the proof of the theorem using the lemmas above. As $p_{\eta}\ge \frac{c }{\eta (2+ \log \eta)^2}$,
		\begin{align} 
		p(\b) \geq \frac{\log \eta}{N} &+\frac{1}{N} \log \frac{  C_{L,\eps}  }{(2 N^2)^{1+\eps} }  + \frac{1+\eps}{N}\left(  \log c  - \log \eta  - 2 \log (2 + \log \eta) \right)\\ 
		&- \frac{1+\eps}{N} \log 2 + \frac{\log \bP \left[ \mathcal{A}_N\right]}{N}.
		\end{align}
		Since $\eta\in [1/2, e^{C_{\beta}N}]$ we obtain, 
		\begin{align} 
		p(\b) \geq - \eps C_\b &+ \frac{1}{N}  \log \frac{  C_{L,\eps}  }{(2 N^2)^{1+\eps} } + \frac{1+\eps}{N}\left(  -\log c - 2 \log (2 + C_\b N) \right)\\ 
		&- \frac{1+\eps}{N} \log 2 + \frac{\log \bP \left[ \mathcal{A}_N\right]}{N}.
		\end{align}
		which can be made arbitrarily small by choosing $\eps$ sufficiently small and $N$ sufficiently large.
	\end{proof}

	\begin{proof}[Proof of Lemma 2.1]
		Notice that
		\begin{equation}
		\begin{split}
		\bP \left[ \mathcal{A}_N\right] &\geq \bP \left[|X_1| < N,...,|X_{N-2}| < N, S_{N-1}=0  \right] \\
		&= \bE \left[\1{|X_1| < N}...\1{ X_{N-2} < N} \bE^{S_{N-2}}\left[X_{1}=0 \right] \right] \\
		&\geq  \frac{C_{L,\eps}}{N^{2+\eps}}\bE\left[\1{|X_1| < N}...,\1{|X_{N-2}| < N} \right].
		\end{split}
		\end{equation}	
		In the last inequality we use \eqref{SVF} (if the last jump $X_{N-1}$ is smaller than $K$ its probability can be lower bounded by a positive constant). Finally,
		\begin{equation}
		\bP \left[ \mathcal{A}_N\right] \geq \frac{C_{L,\eps}}{N^{2+\eps}} \left( 1 - \bP\left[X_1 \geq N \right]\right)^{N-2},	
		\end{equation}
		which implies
		\begin{equation}
		\frac{\log \bP \left[ \mathcal{A}_N\right]}{N} \geq \frac{1}{N}\log\left(\frac{ C_{L,\eps}}{N^{2+\eps}} \right) + \frac{N-2}{N} \log \left( 1 - \bP\left[X_1 \geq N \right]\right), 
		\end{equation}
		which converges to $0$ as $N \to \infty$.
	\end{proof}
	
	\begin{proof}[Proof of Lemma 2.2]
		Let us denote by $\tilde{W}$ a random variable that has the same distribution as $\tilde{W}_N(i,j)$. Notice that, as $\E \tilde{W} = 1$, 
		\begin{equation}\label{pgood0}
		\frac{1}{2} \E \tilde{W} \leq \E\left[ \tilde{W} \1{\tilde{W} > \frac{1}{2} \E \tilde{W}} \right] \leq \sum_{n=0}^{\infty} \E\left[ 2^n  \1{2^{n-1} \leq \tilde{W} }  \right].
		\end{equation}
		Using the fact that $\sum_{n=0}^{\infty} \frac{1}{(n+1)^2} = \frac{\pi^2}{6}$, we obtain
		\begin{equation}
		\frac{3}{\pi^2} \sum_{n=0}^{\infty} \frac{1}{(n+1)^2} \leq  \sum_{n=0}^{\infty} \E\left[ 2^n  \1{2^{n-1} \leq \tilde{W} }  \right].
		\end{equation}
		We now may define $n_0 \geq 0$, as the smallest integer such that 
		\begin{equation} \label{up}
		\frac{3}{\pi^2}\frac{1}{(n_0+1)^2} \leq 2^{n_0} \P\left[ 2^{n_0-1} \leq \tilde{W} \right]. 
		\end{equation}
		Letting $\eta=2^{n_0-1}$ this implies that,
		\begin{equation}\label{pgoodd}
		p_{\eta} \geq \frac{3}{\pi^2}\frac{ 2^{-n_0}}{(n_0+1)^2}.  
		\end{equation}
		On the other hand, by computing the second moment of $ \tilde{W}$ and considering $S'$ as and independent copy of  $S$ we obtain
		\begin{equation}
		\begin{split}
		\E\left[ \tilde{W}^2 \right] &= \E\left[ \bE^{\otimes 2} \left[\exp \left( \sum_{k= 0}^{N-1}  \b \w_{i,S_i} + \b \w_{i,S'_i} - 2\l(\b)  \right)  \Bigg| S, S' \in \mathcal{A}_N  \right] \right] \\ 
		&=  \bE^{\otimes 2}\left[  \exp \left( \left(\l(2\b)- 2\l(\b)\right)  \sum_{k= 0}^{N-1} \1{S_i = {S'}_i}  \right) \1{S, S' \in \mathcal{A}_N } \right] \bP \left[ \mathcal{A}_N\right]^2 \\
		&\leq \exp\left(\left(\l(2\b)- 2\l(\b)\right) N\right),
		\end{split}
		\end{equation} 
		We lower bound the expectation above as
		\begin{equation}
		\begin{split} \label{down}
		\E\left[ \tilde{W}^2 \right] \geq  2^{2(n_0-1)}  \P\left[ 2^{n_0-1} \leq \tilde{W} \right].  
		\end{split}
		\end{equation} 
		By Equations \eqref{up} and \eqref{down} we get
		\begin{equation}
		\frac{3}{\pi^2}\frac{1}{(n_0+1)^2} \leq 2^{n_0} \exp(\left(\l(2\b)- 2\l(\b)\right) N) 2^{-2(n_0-1)},
		\end{equation}
		which implies
		\begin{equation}
		\frac{2^{n_0}}{(n_0+1)^2} \leq \frac{4 \pi^2}{3} \exp(\left(\l(2\b)- 2\l(\b)\right) N).
		\end{equation}
		Let $N_0  \in \N$ be such that if $n > N_0$ then $(3/2)^{n} \leq \frac{2^{n}}{(n+1)^2}$. Then either $n_0 \leq N_0 \leq N$ by taking $N$ sufficiently large, or
		\begin{equation}\label{frac}
		\frac{n_0}{N}\log (3/2) \leq \l(2\b)- 2\l(\b) + \frac{\log (4 \pi^2 /3)}{N},
		\end{equation}
		which finishes the proof of the lemma.
	\end{proof}
	
	\section{Strong Disorder for small temperature}
	In this Section we show that under the assumptions of Theorem \ref{T2}, the polymer measure has a strong disorder phase, for large enough $\b$. Along with Theorem
	\ref{T1}, this allows us to construct a family of polymer measures in which there is a strong disorder phase with $p(\b) = 0$.
	\subsection{Size Biasing}
	Notice that since
	\begin{equation}
	\E\left[W^{\b, \w}_N \right] = 1,
	\end{equation}
	there exists a well defined probability measure $\tilde{\P}_N^{\b}$, called the \textbf{size biasing} measure, absolutely continuous with respect to $\P$ such that
	\begin{equation}
	\frac{\text{d}\tilde{\P}_N^{\b}}{\text{d} \P} = W^{\b, \w}_N.
	\end{equation}
	The following result states that a sequence of positive, mean one random variables, converges to $0$, if and only if it converges to infinity, in probability, under the size biazed distribution. This gives us a condition for which strong disorder holds, in terms of the size biasing measure.
	\begin{lemma} \label{3}
		Let $\left\lbrace W_1, W_2,...\right\rbrace $ be a sequence of positive random variables with $\E\left[X_N \right]= 1$ for all $N$. The following are equivalent:
		\begin{itemize}
			\item \begin{equation}
			\lim_{N \to \infty} W_N = 0,
			\end{equation}	
			$\P-$a.s. 
			\item For all $L > 0$,
			\begin{equation}
			\lim_{N \to \infty} \tilde{\P}_N \left[W_N \geq L \right] = 1.
			\end{equation}
			where the size biased measure $\tilde{\P}$ is defined by its Radon-Nikodym derivative with respect to $\P$,
			\begin{equation}
			\frac{\text{d}\tilde{\P}_N}{\text{d} \P} = W_N.
			\end{equation}
		\end{itemize}
	\end{lemma}
	\begin{proof}
		See \cite[Proposition 4.2]{oriented}.
	\end{proof}
	As in \cite{Lacoin}, we give the following description of the size-biasing measure. Consider an i.i.d. set of random variables $\tilde{\w} = \{\tilde{\w}_{n,z}: n \in \N,z \in \Z\}$ from a probability space $ (\tilde{\Lambda}, \tilde{\cF}, \tilde{\P})$ of distribution given by:
	\begin{equation} \label{quenched}
	\tilde{\P}\left( \tilde{\w}_{1,0} \in \cdot \right) = \E\left[ e^{\b \w_{1,0} - \l(\b)} \1{\w_{1,0} \in \cdot} \right].
	\end{equation}
	For a fixed path $S$, and a given realization of the environments $\{\w_{n,z}: n \in \N,z \in \Z\}$ and $\{\tilde{\w}_{n,z}: n \in \N,z \in \Z\}$, we define $\{\hat{\w}_{n,z}^S: n \in \N,z \in \Z\}$ as: 
	\begin{equation} \label{tilted}
	\hat{\w}_{i,z}^S := \w_{i,z}\1{z \not= S_i} + \tilde{\w}_{i,z}\1{z = S_i}.
	\end{equation}
	One can see that for any bounded continuous function $F: \R \to \R$,
	\begin{equation}
	\tilde{\E}_N^{\b}\left[F(\w) \right] = \bE \otimes \E \otimes \tilde{\E} \left[F(\hat{\w}^S)\right],  
	\end{equation}
	as the change of measure induced by the density,
	\begin{equation}\label{func}
	\frac{\text{d}\tilde{\P}_N^{\b,S}}{\text{d} \P} = \exp\left(\sum_{k= 1}^N \b \w_{i, S_{i}} - \l(\b) \right), 
	\end{equation}
	retains the independence of the elements of the environment but tilts the distribution of the ones that belong to the graph of $S$ by a factor of $\exp\left(\b \w - \l(\b) \right) $. This implies that, given Lemma \ref{3}, the following is sufficient to prove Theorem 2.
	\begin{proposition}
		If
		\begin{equation}
		-\sum_{n\geq 1} K(n) \log K(n) < \b \l'(\b) - \l(\b)
		\end{equation}
		for some $\b > 0$ (in particular, $-\sum_{n\geq 1} K(n) \log K(n) < \infty$), then sequence $\left\lbrace W^{\b, \hat{\w}^S}_N \right\rbrace_{N \geq 1} $ converges to infinity $\bP \otimes \P \otimes \tilde{\P}$-a.s.
	\end{proposition}
	\begin{proof}
		Let us write 
		\begin{equation}
		W^{\b, \hat{\w}^S}_N = \bE' \left[ \exp\left(\sum_{k= 1}^N \b \hat{\w}_{k,S'_k}^S - \l(\b) \right)  \right], 
		\end{equation}
		where $\left(\bP', S' \right) $ is an independent copy of $\left(\bP, S \right)$. Then
		\begin{align}
		W^{\b, \hat{\w}^S}_N &\geq \bP'\left[S_1' = S_1,...,S_N' = S_N \right] \exp\left(\sum_{k= 1}^N \b \hat{\w}_{k,S_k}^S - \l(\b) \right) \\
		&= \prod_{k=1}^N K(X_k) \exp\left(\b \tilde{\w}_{k,S_k} - \l(\b)  \right),  
		\end{align}
		Since it suffices to show that $\lim_N \log W^{\b, \hat{\w}^S}_N \to \infty$, we are left with proving that
		\begin{equation}\label{fina}
		\lim_{N \to \infty} \sum_{k=1}^N \left( \log K(X_k) + \b \tilde{\w}_{k,S_k} - \l(\b) \right) \to \infty,
		\end{equation}
		$\bP \otimes \tilde{\P}$-a.s. Notice that, if $ h_\b:= \bE \otimes \tilde{\E} \left[\log K(X_k) + \b \tilde{\w}_{k,S_k} - \l(\b) \right] > 0$, then \eqref{fina} is a consequence of the Law of Large Numbers, applied to the i.i.d. sequence $\left\lbrace \log K(X_k) + \b \tilde{\w}_{k,S_k} - \l(\b) \right\rbrace_{k \geq 1}$ as
		\begin{equation}
		\lim_{N \to \infty} \frac{\sum_{k=1}^N \left( \log K(X_k) + \b \tilde{\w}_{k,S_k} - \l(\b) \right) }{N} = h_\b,
		\end{equation} 
		$\bP \otimes \tilde{\P} $-a.s. This is a direct consequence of the assumption of the proposition as 
		\begin{equation}
		\begin{split}
		\bE \otimes \tilde{\E} &\left[\log K(X_k) + \b \tilde{\w}_{k,S_k} - \l(\b) \right] = 
		\bE \left[\log K(X_k) \right] +  \tilde{\E}\left[\b \tilde{\w} - \l(\b) \right] \\
		&= \sum_{n \in \Z}\left(  K(n) \log K(n) \right) +  \b \l'(\b) - \l(\b).
		\end{split}
		\end{equation}
	\end{proof}

	\section{No strong disorder case}
	In this section, we prove Theorem \ref{T3}. In the proposition below we use the size biased measure description from the previous section to show that the sequence $\{W_N\}$, under the sized biased measure, is tight (which is equivalent to proving that $\{W_N\}$ is uniformly integrable). This proves that weak disorder holds at every temperature, under the conditions on the increments distribution. 
	\begin{proposition}
		Under the conditions of Theorem 1.3, 	
		\begin{equation} \label{6}
		\lim_{N \to \infty} \bP \P \otimes \tilde{\P} \left[W^{\b, \hat{\w}}_N \geq L \right] \not= 1,
		\end{equation}
		for some $L$ sufficiently large.	
	\end{proposition}
	\begin{proof}
		The idea for this proof is to fix a path $S$ and average with respect to the other variables, then show that the resulting sequence is uniformly bounded. By Markov's Inequality and Fubini's Theorem we have,
		\begin{equation}
		\begin{split}
		\P \otimes \tilde{\P} \left[W^{\b, \hat{\w}}_N \geq L \right] &\leq \frac{1}{L}  \E \otimes \tilde{\E} \left[\bE'\left[\exp\left(\sum_{n=1}^N \b \left( \w_{i,S'_i}\1{S_i \not= S'_i} + \tilde{\w}_{i,S'_i}\1{S_i = S'_i} \right) -\l(\b)  \right)  \right]   \right] \\ 
		& = \frac{1}{L} \bE' \left[ F(\b)^{|S_1^N \cap {S'}_1^N|}\right], 
		\end{split}
		\end{equation}
		where we write
		\begin{equation}
		W_N^{\b, \hat{\w}}= \bE'\left[\exp\left(\sum_{n=1}^N\left(  \b \w_{i,S'_i}\1{S_i \not= S'_i} + \tilde{\w}_{i,S'_i}\1{S_i = S'_i}\right)  -\l(\b)  \right)  \right] 
		\end{equation}
		with $\left(\bP', S' \right) $, an independent copy of $\left(\bP, S \right)$,  
		$F(\b):= \tilde{\E} \left[\exp\left(\b \tilde{\w} -\l(\b)\right)\right]$ and 
		\begin{equation}
		S_m^n := \left\lbrace (i,S_i) : m \leq i \leq n  \right\rbrace.
		\end{equation}
		Notice that it suffices to show that there exists some constant $K_\infty >0$ such that
		\begin{equation} \label{7}
		\bP \left[ \bE' \left[ F(\b)^{|S_1^N \cap {S'}_1^N|}\right] \leq K_\infty \right] \geq 1/2,
		\end{equation}
		since we might have
		\begin{equation}
		\bP \P \otimes \tilde{\P} \left[Z^{\b, \hat{\w}}_N \geq L \right] \leq \frac{ K_\infty}{L} + 1/2,
		\end{equation}
		which proves \eqref{6} by taking $L$ large enough. On the other hand, by considering the last time the paths $S$ and $S'$ intersect, we have
		\begin{equation}
		\begin{split} \label{8}
		\bE' \left[ F(\b)^{|S_1^N \cap {S'}_1^N|}\right] &= \sum_{n=0}^N \bE' \left[ F(\b)^{|S_1^n \cap {S'}_1^n|}  \1{S'_n = S_n} \1{S_{n+1}^N \cap {S'}_{n+1}^N = \emptyset }  \right]\\ 
		&\leq \sum_{n=0}^N F(\b)^n \bP' \left[S_n = S'_n \right]
		\end{split}
		\end{equation}
		In the second line, we use that $F(\b) \geq 1$. In fact, as we mentioned before, if the distribution of $\w$ is unbounded, $F(\b) \to \infty$ as $\b \to \infty$ (see Lemma \ref{sup}). To finish the proof we use two lemmas stated below. The first one is Theorem 2.1 from \cite{unimod} and states that the independent sum of two symmetric unimodal distributions is again unimodal. This implies that the distribution of $S_n$ is also symmetric and unimodal and that 
		\begin{equation}
		\bP' \left[S_n = S'_n \right] \leq \frac{1}{|S_n|}.
		\end{equation}
		In the second lemma below, we show that $S_n$ grows faster that any exponential, eventually almost surely. Here we use the crucial fact that $\a > 1$, the lemma being false otherwise. This implies that there exists $K_S > 0$, that might also depend on $\b$, such that
		\begin{equation} \label{9}
		\sum_{n=0}^\infty F(\b)^n \frac{1}{|S_n|} < K_S.
		\end{equation}
		This is sufficient to obtain \eqref{7} and conclude the proof of the theorem.
	\end{proof}
	\begin{lemma} \label{l1}
		Given $\left\lbrace X_1,X_2,...  \right\rbrace $ i.i.d. integer valued random variables, and the distribution of $X_1$ being unimodal and symmetric, then the distribution of $S_n = X_1 + ...+ X_n$ is also unimodal and symetric and 
		\begin{equation} \label{unimod2}
		\bP \left[S_n = x \right] \leq \frac{1}{|x|},
		\end{equation}
		for any $x \in \Z$.
	\end{lemma}
	\begin{proof}
		In \cite{unimod}, they show that the sum of two independent unimodal and symmetric random variables is also unimodal and symmetric. For \eqref{unimod2}, notice that
		\begin{equation}
		1 \geq \sum_{0 \leq y \leq x} \bP \left[S_n =  y \right] \geq x \bP \left[S_n =  x \right]. 
		\end{equation} 
	\end{proof}
	\begin{lemma} \label{l2}
		Given $\left\lbrace X_1,X_2,...  \right\rbrace $ i.i.d. integer valued random variables and assuming that the distribution of $X_1$ satisfies
		\begin{equation} \label{c1}
		\bP\left[ X_1 \geq n \right]   \geq \frac{C (\log \log n)^\alpha}{ \log n},
		\end{equation}
		for $\alpha > 1$, $C > 0$ and
		\begin{equation}
		\frac{\bP\left[X_1 \in (s_n, 2n s_n) \right]}{\bP\left[X_1 \geq s_n \right]} \leq \frac{1}{n^\g},
		\end{equation}
		where $\g > \dfrac{1}{2}$ and
		\begin{equation}
		s_n := \min \left\lbrace s  \in \N: \bP\left[X_1 \geq s \right] \leq \frac{\left(\log n \right)^2 }{n}\right\rbrace, 
		\end{equation}
		for all $n$ sufficiently large, then for all constant $K > 0$, 
		\begin{equation} \label{10}
		|S_n| > K^n,
		\end{equation}
		for all paths $S$, eventually for all $n$ large enough, $\bP$-a.s.	
	\end{lemma}
	\begin{proof}
		Given the first increments $\left\lbrace X_1,...,X_n\right\rbrace $, let $X^{(n)}_n$ and $X^{(n-1)}_n$ be the highest and second highest values among $\left\lbrace |X_1|,...,|X_n|\right\rbrace $. The proof of the lemma relies on two facts: the maximum $X^{(n)}_n$ satisfies \eqref{10}, i.e., 
		\begin{equation} \label{11}
		X^{(n)}_n > K^n,
		\end{equation}
		eventually $\bP-a.s.$, and that $X^{(n)}_n$ and $S_n$ have roughly the same order, since
		\begin{equation} \label{12}
		X^{(n-1)}_n \leq \frac{1}{2n} X^{(n)}_n,
		\end{equation}
		for some constant $\d > 0$, eventually $\bP-a.s.$ In fact, inequalities \eqref{11} and \eqref{12} imply,
		\begin{equation} \label{ineq2}
		|S_n| \geq X^{(n)}_n - (n-1) X^{(n-1)}_n \geq X^{(n)}_n - \frac{(n-1)}{2n}X^{(n)}_n\geq \frac{1}{2} K^n,
		\end{equation} 
		eventually $\bP-a.s.$ To show \eqref{11}, observe that
		\begin{equation}
		\bP \left[ X^{(n)}_n \leq K^n \right] = \left( 1 - \bP \left[X_1 > K^n \right]\right) ^n \leq \left( 1 - \frac{C_K (\log n)^\alpha}{n}\right)^n \leq e^{-C_K (\log n)^\alpha},   
		\end{equation}
		for some constant $C_K >0$, which by Borel-Cantelli's Lemma, implies \eqref{11}. For \eqref{12} we have that, for $s \leq t \in \N$
		\begin{equation}
		\bP\left[X^{(n-1)}_n = s,  X^{(n)}_n = t \right] \leq {n \choose 2} \bP\left[X_1 \leq s \right]^{n-2} \bP\left[X_1 = s \right] \bP\left[X_1 = t \right].    
		\end{equation}
		Then,
		\begin{equation}
		\bP\left[X^{(n-1)}_n > \frac{1}{2n} X^{(n)}_n \right] \leq \sum_{s=0}^\infty {n \choose 2} \bP\left[X_1 \leq s \right]^{n-2} \bP\left[X_1 = s \right] \bP\left[X_1 \in (s, 2ns) \right].    
		\end{equation}
		We split the last sum into two parts. The sum up to $s = s_n - 1$ can be bounded by
		\begin{equation}
		\sum_{s=0}^{s_n - 1} {n \choose 2} \left( 1 - \frac{\left( \log n \right)^2 }{n} \right)^{n-2} \bP\left[X_1 = s\right] \leq C' n^2 e^{-(\log n)^2} . 	
		\end{equation}
		The second part of the sum can be bounded by
		\begin{equation} \label{ineq}
		\sum_{s= s_n }^{\infty} {n \choose 2} \bP\left[X_1 = s \right] \bP\left[X_1 \in (s_n, 2n s_n) \right] \leq \frac{n^2 \bP\left[X_1 \geq s_n \right]^2 }{n^\g} \leq \frac{(\log n)^4}{n^\g}. 
		\end{equation}
		Unfortunately, the last inequality is not enough to directly conclude \eqref{11} by Borel-Cantelli's Lemma, as $\g$ might be smaller or equal that $1$. In order to overcome this, let us consider $\left\lbrace U_i: i \in \N \right\rbrace$ a sequence of independent, Uniform-$[0,1]$ random variables, and let us couple the i.i.d. sequence $\left\lbrace |X_1|,|X_2|,...  \right\rbrace $ with the sequence $\left\lbrace F^{-1}(U_1),F^{-1}(U_2),...  \right\rbrace $, where $F$ is the cumulative distribution function $F(x) := \bP \left[  |X_1| \leq x \right] $ and $F^{-1}$, the generalized inverse distribution function, defined as
		\begin{equation}
		F^{-1}(p)=\inf\{x\in \mathbb {R} :F(x)\geq p\},
		\end{equation}
		for $p\in [0,1]$. As before, let us denote by $U^{(n)}_n$ and $U^{(n-1)}_n$, the highest and second highest values among $\left\lbrace U_1,...,U_n\right\rbrace $. 
		In particular, this implies that $X_n^{(n)} = F^{-1}(U^{(n)}_n)$ and $X_n^{(n-1)} = F^{-1}(U^{(n-1)}_n)$. Consider the random variable $\t_n^{(1)}$, as the first time after $n$, such us the second maximum $U^{(k-1)}_k$ needs to be updated, i.e., 
		\begin{equation}
		\t_n^{(1)}:= \min \left\lbrace k > n: U_k > U^{(n-1)}_n \right\rbrace.
		\end{equation}
		and  analogously, let $\t_n^{(2)}$ be the second time after $n$, such us the second maximum is updated:
		\begin{equation}
		\t_n^{(2)}:= \min \left\lbrace k > \t_n^{(1)}: U_k > U^{(\t_n^{(1)}-1)}_{\t_n^{(1)}} \right\rbrace.
		\end{equation}
		Define the events
		\begin{equation}
		\mathcal{B}:= \left\lbrace X^{(n)}_n > K^n, \text{ for all } n \text{ sufficiently large}\right\rbrace, 
		\end{equation}
		\begin{equation}
		\mathcal{C}:= \left\lbrace X^{(n^2-1)}_{n^2} \leq \frac{1}{2n^2} X^{(n^2)}_{n^2},\text{ for all } n \text{ sufficiently large}\right\rbrace, 
		\end{equation}
		and
		\begin{equation}
		\mathcal{D} := \left\lbrace \t_{n^2}^{(2)} \geq (n+1)^2,\text{ for all } n \text{ sufficiently large} \right\rbrace. 
		\end{equation}
		We show that on the intersection of the three events, Inequality \eqref{10} holds. In fact, on $\mathcal{B} \cap \mathcal{C}$ we have that $|S_{n^2}| \geq  \frac{1}{2} K^{n^2}$ for all $n$ large enough. By intersecting with the event $\mathcal{D}$ we have that between $n^2$ and $(n+1)^2$, the second maximum is updated at most one time. This implies that for all $t \in (n^2, (n+1)^2)$, the pair $\left(X_t^{(t-1)},X_t^{(t)} \right) $ is either $\left(X_{n^2}^{(n^2-1)},X_{n^2}^{(n^2)} \right)$ or $\left(X_{(n+1)^2}^{((n+1)^2-1)},X_{(n+1)^2}^{((n+1)^2)} \right)$. If it is equal to $\left(X_{n^2}^{(n^2-1)},X_{n^2}^{(n^2)} \right)$ we have that
		\begin{equation}
		|S_{t}| \geq X_{n^2}^{(n^2)} - (t-1)X_{n^2}^{(n^2-1)} \geq  X_{n^2}^{(n^2)} - \frac{(n+1)^2-1}{2n^2}X_{n^2}^{(n^2)} \geq \frac{1}{3} K^t.
		\end{equation}  
		If it is equal to $\left(X_{(n+1)^2}^{((n+1)^2-1)},X_{(n+1)^2}^{((n+1)^2)} \right)$ we obtain
		\begin{equation}
		|S_{t}| \geq  X_{(n+1)^2}^{((n+1)^2)} - (t-1)X_{(n+1)^2}^{((n+1)^2-1)} \geq  X_{(n+1)^2}^{((n+1)^2)} - \frac{(n+1)^2-1}{2(n+1)^2}X_{(n+1)^2}^{((n+1)^2)} \geq \frac{1}{2} K^t.
		\end{equation}
		To finish the proof of the lemma we verify that the three events have probability one.  $\bP\left[\mathcal{B} \right]=1 $ was already shown in \eqref{ineq2}. As $\g > 1/2$, the upper bound obtained in \eqref{ineq} suffices to obtain that $\bP\left[\mathcal{C} \right]=1 $. As for $\mathcal{D}$ we have
		\begin{equation}
		\begin{split}
		\bP\left[ \t_{n^2}^{(2)} < (n+1)^2 \right] &\leq \bP\left[\exists i, j \in (n^2, (n+1)^2): U_i > U_{n^2}^{(n^2-1)}, U_j > U_{n^2}^{(n^2-1)} \right] \\
		&\leq (2n)^2  \bP\left[ U_{n^2+1} > U_{n^2}^{(n^2-1)} \right]^2 
		\end{split}
		\end{equation}
		Since $ U_{n^2+1}$ and $U_{n^2}^{(n^2-1)}$ are independent, their joint distribution can be computed explicitly \cite{order}. This yields
		\begin{equation}
		\bP\left[ U_{n^2+1} > U_{n^2}^{(n^2-1)} \right] = \int_0^1 \int_0^1  n^2(n^2-1)u^{n^2-2}(1-u) \1{v>u} dv du = \frac{2}{n^2+1},
		\end{equation}
		which proves the lemma.
\end{proof}
 \appendix
\section{Properties of the exponential moment}

\begin{lemma} \label{sup}
	Let $\w$ be a random variable with $\E\left[e^{\b \w} \right] < \infty$ for all $\b > 0$, $s = \text{ess sup }\w$ and $\l(\b) := \log \E\left[e^{\b \w} \right]$.
	\begin{itemize}
		\item[a)] Let $K < s$, then
		\begin{equation}
		\lim_{\b \to \infty} \tilde{\P}^\b \left[\tilde{\w} < K \right] \to 0,
		\end{equation}
		\item[b)] $\lim_{\b \to \infty} \l'(\b) = s$,
		\item[c)] $\lim_{\b \to \infty} \b \l'(\b) - \l(\b)= -\log \P\left[\eta = s\right].$
	\end{itemize}
\end{lemma}
\begin{proof}
The idea for this proof is that the sequence of measures $\P^\b$, induced by the density
\begin{equation}\label{func1}
\frac{\text{d}\tilde{\P}^{\b}}{\text{d} \P} = e^{\b \w - \l(\b)}, 
\end{equation}
tend to put all the mass on the essential supremum $s$ of $\w$, as $\b \to \infty$.	
	\begin{itemize}
		\item[a)] Let $K' > 0$ such that $K < K'< s$. Then $\P\left[ \w > K' \right] = \d > 0$ and
		\begin{equation}
		\tilde{\P}^\b \left[\tilde{\w} < K \right] = \frac{\E\left[e^{\b \w} \1{\w < K}\right] }{\E\left[e^{\b \w}\right]} \leq \frac{e^{\b K} }{\E\left[e^{\b \w} \1{\w > K'}\right]} \leq \frac{e^{-\b(K'-K)}}{\d} \to 0,
		\end{equation}
		as $\b \to \infty$.
		\item [b)]Assume $s < \infty$ for the rest of the proof. The case $s = \infty$ is analogous. Fix $\eps >0$. Let $K>0$ such that $0 < s- K < \eps$. Notice that $\l'(\b) = \tilde{\E}^\b \left[\tilde{\w} \right]  = \E\left[\w e^{\b \w - \l(\b)} \right] \leq s  $ and
		\begin{equation}
		\tilde{\E}^\b \left[\tilde{\w} \right] \geq \tilde{\E}^\b \left[\tilde{\w} \1{ \tilde{\w} \geq K}\right] \geq (s-\eps) \P\left[\tilde{\w} \geq K \right] \geq (s-\eps)(1-\eps),
		\end{equation}
		for some large enough $\b$, by the previous item. This proves b).
		\item[c)] Given $\eps > 0$, let $K > 0$ such that $0 < s- K < \eps$ and $\P\left[ \w \geq K \right] \leq \P\left[ \w = s \right] + \eps.$ Then, 
		\begin{equation}
		\begin{split}
		e^{\b s  - \l(\b)} &\left( \P\left[ \w = s \right] + \eps \right)   \geq  \E\left[e^{\b \w - \l(\b)} \1{\w \geq K}\right] = \tilde{\P}^\b\left[\tilde{\w} \geq K \right] \\ &\geq e^{\b K - \l(\b)} \P\left[ \w \geq K\right]  \geq e^{\b(s - \eps)  - \l(\b)}\P\left[ \w = s \right]. 
		\end{split}
		\end{equation}
		Applying logarithms and taking $\b \to \infty$ we obtain
		\begin{equation}
		\lim_{\b \to \infty} \b s - \l(\b) = -\log \P\left[ \w = s \right], 
		\end{equation}
		which proves c).
	\end{itemize}
\end{proof}

\end{document}